\documentclass[12pt]{article}

\usepackage{amsmath,amsthm}
\usepackage{amsfonts,amsrefs}
\usepackage{verbatim}
\usepackage{amssymb}
\usepackage{txfonts}
\usepackage{amscd}
\usepackage{bbm}
\usepackage{hyperref}
\usepackage{mathrsfs}
\usepackage{fouriernc}

\usepackage{tikz-cd}
\usepackage{gauss}
\usepackage{cleveref}

\newtheorem{theorem}{Theorem}[section]
\newtheorem{corollary}[theorem]{Corollary}
\newtheorem{lemma}[theorem]{Lemma}
\newtheorem{proposition}[theorem]{Proposition}
\newtheorem{definition}[theorem]{Definition}

\newtheorem{remark}[theorem]{Remark}

\newcommand{\RR}{\mathbb{R}}
\newcommand{\ZZ}{\mathbb{Z}}

\newcommand{\NN}{\mathbb{N}}

\newcommand{\tr}{\operatorname{tr}}

\newcommand{\Mat}{\mathbb{M}}
\newcommand{\SLZ}[1][3]{\operatorname{SL}_{#1}(\ZZ)}

\usepackage{authblk}
\title{Spectral gap for the cohomological Laplacian of $\SLZ$}

\author[1]{Marek Kaluba\thanks{marek.kaluba@kit.edu} }
\author[2]{Piotr Mizerka \thanks{pmizerka@impan.pl}}
\author[3]{Piotr W. Nowak \thanks{pnowak@impan.pl}}
\affil[1]{Karlsruher Institute f\"{u}r Technologie, Karlsruhe, Germany}
\affil[2,3]{Institute of Mathematics of the Polish Academy of Sciences, Warsaw, Poland}

\begin{document}

\maketitle

\abstract{We show that the cohomological Laplacian in degree 1 in the group cohomology of $\SLZ$
is a sum of hermitian squares in the algebra $\Mat_n(\RR G)$. We provide an estimate
of the spectral gap for this Laplacian for every unitary representation. Our results have been obtained by convex optimization technique of semidefinite programming (SDP).}

\section{Main results}
Kazhdan's property $(T)$ has numerous characterizations,
two of which form the basis of our current study.
The first one describes groups with property $(T)$ as exactly those,
whose first cohomology $H^1(G,\pi)$ with coefficients in any unitary representation $\pi$ vanishes.
The second one is the algebraic characterization obtained by N.~Ozawa \cite{TakaOzawa}
who proved that for a finitely generated group $G$, property $(T)$
is equivalent a sum of squares decomposition of $\Delta^2-\lambda\Delta$
in the real group ring $\RR G$. In \cite{KNO2019} and \cite{KKN2021} this characterization was used to show that $\operatorname{Aut}(F_n)$, the automorphism group of the free group 
$F_n$ on $n$ generators, has property $(T)$ for $n\ge 5$.

Group cohomology provides an appropriate framework to extend property $(T)$ to higher dimensions
and in this work we are concerned with such generalizations.
In this spirit we consider the vanishing of higher cohomology $H^n(G,\pi)$
with coefficients in an arbitrary unitary representation $\pi$, see \cite{BaderNowak2020},
as well as algebraic conditions in the algebra $\Mat_n(\RR G)$
of matrices with group ring coefficients, that imply such vanishing.

Our main goal is to show that an algebraic condition, equivalent to the fact that
$H^1(G,\pi)$ vanishes and $H^2(G,\pi)$ is reduced for every unitary representation $\pi$,
is explicitly satisfied for the group $G=\SLZ$.
More precisely, we show the following

\begin{theorem}\label{theorem:main}
Let $\Delta_1$ be the cohomological Laplacian for $\SLZ$.
Then $\Delta_1-\lambda I$ is a sum of squares in $\Mat_n(\RR \SLZ)$
for some $\lambda > 0$.

In particular, for every unitary representation $\pi\colon \SLZ\to U(\mathcal{H})$ on a Hilbert space
the cohomological Laplacian $\pi(\Delta_1) \in B(\mathcal{H})$ is invertible
and has a spectral gap at least $\lambda$.

For $\Delta_1$ of the presentation complex we have $\lambda \geqslant 0.32$.
\end{theorem}

To our best knowledge this is the first explicit estimate of the spectral gap
of the cohomological Laplacian for all unitary representations and it can be viewed as
a higher-dimensional analogue of the Kazhdan constant.
Another consequence of the above statement is that
the cohomology $H^2(\SLZ,\pi)$ is always reduced for any unitary representation $\pi$.
We acknowledge that soon after the completion of our work
U.~Bader and R.~Sauer announced a much broader result with a similar conclusion.

Our approach is based on the use of Fox derivatives to compute an explicit matrix
representative of the cohomological Laplacian in degree $1$.
We express the Laplacian element for the Steinberg relations in $G=\SLZ$
as a matrix $\Delta_1 \in \Mat_6(\RR G)$. Then we employ semidefinite programming (SDP) to find a sum of squares decomposition
for the expression $\Delta_1 - \lambda I_6$. The computations showed that $\lambda$ can be chosen to be positive in this expression which constitutes the desired positive spectral gap.
This suffices to prove the main result.

The article is organized as follows.
\Cref{section:FoxCalculus} contains basic notions of the Fox calculus and
direct computations of the differentials for the Steinberg presentation of $\SLZ$.
In \cref{section:coneOfSoS} we introduce the cone of sum of squares of matrices
over a group ring and show that the identity matrix is an order unit for it.
We describe a reduction of the existence of a sum of squares decomposition to
a problem of semidefinite optimization in \cref{section:SoSandOptimization}.
We finish by briefly describing \emph{certification} which transforms
an inexact sum of squares decomposition obtained through numerical optimization
to a mathematically rigorous lower bound on the spectral gap.
The results of the optimization enable us to show \cref{corollary:certifiedDelta_1}
from which \cref{theorem:main} follows.

\subsection*{Acknowledgements}

The first author was supported by SPP 2026 \emph{Geometry at infinity} funded by the Deutsche Forschungsgemeinschaft.
The second and third authors were supported by the European Research Council (ERC) grant \emph{Rigidity of groups and higher index theory} 
under the 
European Union's Horizon 2020 research and innovation program (grant agreement no. 677120-INDEX).

\section{Fox calculus and 2-cohomology}\label{section:FoxCalculus}

We begin by recalling briefly the Fox calculus and the associated complex computing group cohomology.
We refer to \cites{lyndon,lyndon-schupp} for further details.
Let $G=\langle S \:\vert\: R\rangle$ be a finitely presented group,
with a generating set $S$ and a set of relations $R$.
A family of derivations into the group ring was described by Fox \cites{fox1,fox2}
and is defined by the formulas
\begin{align*}
\dfrac{\partial e}{\partial s_i}&=0,\\
\dfrac{\partial s_i}{\partial s_j} &= 
\begin{cases}
	1 & \text{if } i=j \\
	0 & \text{if } i \neq j
\end{cases}
,\\
\dfrac{\partial (uv)}{\partial s_i}&=\dfrac{\partial u}{\partial s_i} +u \dfrac{\partial v}{\partial s_i},
\end{align*}
For any left $G$-module consider a cochain complex as in \cite{lyndon},
\[
	C^0 \xrightarrow{\quad d_0\quad}
	C^1 \xrightarrow{\quad d_1\quad}
	C^2 \xrightarrow{\quad d_2\quad}
	\dots,
\]
with
\begin{align*}
d_0 & = \begin{bmatrix}
		1-s_1 \\
		\dots \\
		1-s_n
	\end{bmatrix},
	\quad s_i\in S, &\text{and} &&
d_1 &= \left[ \dfrac{\partial r_i}{\partial s_j}\right]_{r_i\in R, s_j\in S}.
\end{align*}
The matrix $d_1$ is also referred to as the \emph{Jacobian} of $G$.
As pointed out by Lyndon \cite{lyndon}, the above complex, after tensoring with a $G$-module $M$,
computes the cohomology of $G$ with arbitrary coefficient module $M$.
The maps $d_0$ and $d_1$ arise also geometrically,
defining the cochain maps for the cohomology of the presentation complex
(which can be chosen to be a $2$-skeleton of $K(G,1)$).
The cohomological Laplacian in degree $1$ is defined as
\begin{equation}\label{eqn:Laplacian}
	\Delta_1= d_0d_0^*+d_1^*d_1.
\end{equation}
It can be easily checked that the matrix $d_1^*d_1$ is of the form
$$\left[ \sum_{k} \left(\dfrac{\partial r_k}{\partial s_i}\right)^* \dfrac{\partial r_k}{\partial s_j}  \right]_{i,j} =  \sum_{k}\left[\left(\dfrac{\partial r_k}{\partial s_i}\right)^* \dfrac{\partial r_k}{\partial s_j}  \right]_{i,j}.$$
We observe that the above sum is in fact a sum of squares which correspond to relations,
namely it is of the form $\sum_r J(r)^* J(r)$, where
$$J(r)=\left[\begin{array}{ccc}\dfrac{\partial r}{\partial s_1} & \dots  & \dfrac{\partial r}{\partial s_n} \\0 & \dots & 0 \\
\vdots &\ddots&\vdots\\
0 & \dots & 0\end{array}\right].$$
Therefore we notice the following lemma.
\begin{lemma}\label{lemma:addingRelations}
Let $G=\langle S\vert R \rangle$ be a finitely presented group and consider a subset $R'\subseteq R$ of the set of relations.
If
$$d_0d_0^* +\sum_{r\in R'} J(r)^* J(r) - \lambda I =\sum \xi^*\xi,$$
then $\Delta_1-\lambda I = \sum \xi^*\xi + \sum_{r \in R\setminus R'} J(r)^*J(r)$ is a sum of squares in $\Mat_n(\RR G)$.
\end{lemma}

\subsection{Steinberg presentation of $\SLZ$}
Let $G=\SLZ$. We use the presentation with six generators given by elementary matrices
(see \cite{1992_sl_3_z_presentations}, for example) to define Fox derivatives for $G$.
Let $E_{i,j}\in G$ ($i\neq j$) be the $(i,j)$-th elementary matrix,
i.e. $E_{i,j} = I + \delta_{i,j}$.
Then, the group $G$ can be presented as follows:
\begin{equation}\label{eqn:SL3Z_presentation}
	G=\left\langle E_{i,j}\;\big|\;r_{i,j,k},r_{i,j,k}',r\right\rangle,
\end{equation}
where
\begin{gather*}
	r_{i,j,k}=\left[E_{i,j},E_{i,k}\right], \quad
	r_{i,j,k}'=\left[E_{i,j},E_{j,k}\right]E_{i,k}^{-1}
\intertext{are the Steinberg relations ($i,j$ and $k$ are distinct), and}
	r = \left(E_{1,2}E_{2,1}^{-1}E_{1,2}\right)^4.
\end{gather*}
The Fox derivatives are as follows:
\begin{gather*}
	\frac{\partial r_{i,j,k}}{\partial E_{i,j}}=1-E_{i,k},\quad
	\frac{\partial r_{i,j,k}}{\partial E_{i,k}}=E_{i,j}-1,\quad
	\frac{\partial r_{i,j,k}'}{\partial E_{i,k}}=-1,\\
	\frac{\partial r_{i,j,k}'}{\partial E_{i,j}}=1-E_{i,k}E_{j,k},\quad
	\frac{\partial r_{i,j,k}'}{\partial E_{j,k}}=E_{i,j}-E_{i,k}, \quad \text{and}\\
	\frac{\partial r_{i,j,k}}{\partial E_{s,t}}=\frac{\partial r_{i,j,k}'}{\partial E_{s',t'}}=0,
\end{gather*}
when $E_{s,t}$ and $E_{s',t'}$ do not appear in the numerator.

\section{The cone of sums of squares in $\Mat_n(\RR G)$}\label{section:coneOfSoS}

Let
\[
\Sigma^2 \Mat_n (\RR G) = \left\{
	\sum_{i=1}^k M_i ^* M_i \colon k\in \NN, M_i \in \Mat_n (\RR G)
\right\}
\]
denote the cone of sums of (hermitian) squares in $\Mat_n (\RR G)$.
Related to the cone there is a natural order on $\Mat_n (\RR G)$, namely we will write $A \geqslant B $
whenever $A-B \in \Sigma^2 \Mat_n (\RR G)$.
For an element of the group ring $m = \sum_g m_g g$ we denote its $\ell_1$-norm by
$\|m\|_1 = \sum_g |m_g|$.
For a matrix $M=\left[m_{i,j}\right]\in \Mat_n (\RR G)$, we define $\ell_1$-norm as
$\left\|M\right\|_1 = \sum_{i,j} \left\| m_{i,j} \right\|_1$.

Let $E\subset G$ be a subset. For a matrix $M\in \Mat_n(\RR G)$ we will write
\[
M \in \Sigma^2_{E} \Mat_n (\RR G)
\]
to denote that $M$ admits a decomposition into a sum of squares
\[
M=\sum_i M_i^*M_i,
\]
where $M_i \in \Mat_n(\RR E) \cong \Mat_n (\RR ) \otimes \RR E$
i.e. the entries of each $M_i$ are supported in $E$.

\begin{definition}
A matrix $u\in \Mat_n(\RR G)$ is an order unit for $\Sigma^2 \Mat_n (\RR G)$ if
for every $*$-invariant $M \in \Mat_n (\RR G)$ there exists $R\geqslant 0$ such that
\[
	M + Ru \in \Sigma^2 \Mat_n (\RR G).
\]
\end{definition}

A quantitative version of the fact that the identity is an order unit in $\Mat_n(\RR G)$
allows to certify algebraic spectral gaps obtained via semidefinite programming (SDP),
see \cref{section:SoSandOptimization}.
\begin{proposition}\label{proposition:positiveOrderUnit}
	The identity matrix is an order unit for $\Sigma^2 \Mat_n (\RR G)$.
	Moreover for every $*$-invariant $M$ we have
	$$M + \left\|M\right\|_1 I_n \in \Sigma^2 \Mat_n (\RR G).$$
\end{proposition}
\begin{proof}
	Let $M\in \Mat_n(\RR G)$ be $*$-invariant. Consider first the case $n=1$. The $(1\times 1)$ matrix $M$ can be then thought of as an element of the group ring $\RR G$. Assume first $M=\pm \left(g + g^{-1}\right)$, where $g\in G$. In this case
	\begin{equation*}
		M+\| M\| _1 = \pm \left(g + g^{-1}\right) + 2 = (1\pm g)^*(1\pm g)
	\end{equation*}
	is a square. More generally, every $*$-invariant element $M = \sum_g c_g g \in \RR G$ can be written as $\sum_g \frac{c_g}{2} \left(g + g^{-1}\right)$, therefore\
	\[
	M + \|M\|_1 = M + \sum_g |c_g | = \sum_g \frac{|c_g|}{2} \left(\pm\left(g + g^{-1}\right) + 2\right) \in \Sigma^2 \RR G.
	\]
	Suppose now $n=2$ and consider first the case when $M=X_g$, $g\in G$, where
	\begin{eqnarray*}\label{eqnarray:2matricesSimpleForm}
		X_g=\begin{bmatrix}0&\pm g\\\pm g^{-1}&0\end{bmatrix}.
	\end{eqnarray*}
	Then one can write $X_g+\frac{1}{2}\| X_g\| _1I_2$ as a square as follows:
	\[
	X_g+\frac{1}{2}\| X_g\| _1I_2=X_g+I_2=
	\frac{1}{2} \left(I_2 + X_g\right)^* \left(I_2 + X_g\right).
	\]
	Using a similar argument as in the case $n=1$ one can show that for
	any $x\in \RR G$ and $X_x=\begin{bmatrix}0&x\\x^*&0\end{bmatrix}$,
	the matrix $X_x+\frac{1}{2}\| X_x\| _1I_2$ is a sum of squares.

	We prove now the general case $n\geqslant 2$. Since $M^*=M$,
	we can decompose $M = \left[m_{i,j}\right]$ as
	$$
	M= \sum_{i,j} M_{i,j},
	$$
	where $M_{i,j} = \frac{1}{2}\left(\delta_{i,j}m_{i,j} + \delta_{j,i}m_{j,i}\right)$ is $*$-invariant ($\delta_{i,j} $ denotes the $n\times n$ real matrix with $1$ at the $(i,j)$-th position and zeroes elsewhere).
	Analogously to $X_x$ we can compute that
	\[M_{i,j} + \left\|M_{i,j}\right\|_1 \underbrace{\left(\delta_{i,i} + \delta_{j,j}\right)}_{I_{i,j}} \in \Sigma^2 \Mat_n (\RR G).\]
	We note that if $i=j$ then $M_{i,j} = \delta_{i,i}m_{i,i}$ and therefore adding $\left\|M_{i,i} \right\|_1 $ is sufficient.
	Then we have
	\begin{multline*}
		M + \|M\|_1 I_n = \sum_{i,j} \left(
			M_{i,j} + \|M_{i,j}\|_1 I_n
		\right) \\
		\geqslant \sum_{i,j} \left(
			M_{i,j} + \|M_{i,j}\|_1 I_{i,j}
		\right) \in \Sigma^2 \Mat_n (\RR G).
		\qedhere
	\end{multline*}
\end{proof}

\subsection{Sums of squares and optimization problems}\label{section:SoSandOptimization}
In this section we reduce the problem of finding a decomposition into sum of squares of
matrices to the existence of a certain positive semi-definite matrix.
We note that the conditions imposed on the matrix are linear, therefore the existence of
such a matrix turns out to be the problem of the feasibility of a semi-definite optimization problem.


Let $G=\langle s_1,\ldots,s_n \:\vert\: R\rangle$ and $E=\{g_1,\ldots,g_m \}\subseteq G$ be a subset
(it is useful to think that $E = B_d$, the metric ball in $G$ of radius $d$). Denote by $v_i$ the column vector of $m$ real coordinates with $1$ at the $i$-th place and zeroes elsewhere and let $v = \sum_j^m g_j \otimes v_j\in\RR G \otimes \RR^m$. Finally, define $\mathbbm{x}$ to be the tensor matrix $I_{n} \otimes v \in \Mat_{n}(\RR )\otimes
\left(\RR G \otimes \RR^m \right)\cong\Mat_{nm\times n}(\RR G)$. For example, if $n=2$ and $m=3$, then
$$
\mathbbm{x}=\begin{bmatrix}
	g_1&0\\
	g_2&0\\
	g_3&0\\
	0&g_1\\
	0&g_2\\
	0&g_3
\end{bmatrix}\in\Mat_{6\times 2}(\RR G).
$$

\begin{lemma}\label{lemma:sdp_matrix_sos}
Matrix $M\in \Mat_n (\RR G)$ admits a sum of squares decomposition in
$E$, i.e. $M\in \Sigma^2_E \Mat_n (\RR G)$
if and only if there exists a positive semi-definite matrix
$P\in \Mat_{nm}(\RR ) \cong \Mat_n (\RR )\otimes \Mat_m(\RR)$ such that
\begin{equation}\label{eqn:x'Px}
  M=\mathbbm{x}^*P\mathbbm{x}.
\end{equation}
\end{lemma}

\begin{proof}
 Suppose first that there exists a positive semi-definite matrix $P\in \Mat_{nm}(\RR ) \cong \Mat_n (\RR )\otimes \Mat_m(\RR)$ such that \eqref{eqn:x'Px} holds. Since $P$ is positive semi-definite there exists a matrix
$Q \in \Mat_n (\RR) \otimes \Mat_m (\RR) $ such that $P = Q^* Q$.
Let us write $Q = \sum_{i,j}^m Q^{i,j} \otimes \delta_{i,j}$
for some $Q^{i,j} \in \Mat_n(\RR)$ and $\delta_{i,j} $ denoting the $m\times m$ real matrix with $1$ at the $(i,j)$-th position and zeroes elsewhere. Then
\[
Q\mathbbm{x} =
\left(
  \sum_{i,j}^m Q^{i,j} \otimes \delta_{i,j}
\right)
\left(
  I_n \otimes \left(
    \sum_j^m g_j \otimes v_j
  \right)
\right) =
\sum_i^m \underbrace{
  \left(\sum_j^m Q^{i,j} g_j\right)
}_{M_i \in \Mat_{n}(\RR E)}
\otimes v_i.
\]
By \eqref{eqn:x'Px} and the definition of $Q$ we can compute
\begin{multline*}
M = \mathbbm{x}^* Q^* Q \mathbbm{x} = \big(Q\mathbbm{x}\big)^* Q\mathbbm{x} \\
= \left(\sum_i^m M_i \otimes v_i \right)^*
\left( \sum_j^m M_j \otimes v_j \right) =
\sum_{i,j}^m M_i^*M_j \otimes v_i^Tv_j = \sum_i^mM_i^*M_i.
\end{multline*}
Assume now that $M=\sum_{i}^{m'} M_i^*M_i$ for some $M_i\in \Mat_{n}(\RR E)$. There exist matrices $Q^{i,j} \in \Mat_n(\RR)$, where $1\leq i\leq m'$ and $1\leq j\leq m$, such that $M_i=\sum_j^{m}Q^{i,j}g_j$. Proceeding analogously as previously, we write $Q = \sum_{i}^{m'}\sum_j^m Q^{i,j} \otimes \delta'_{i,j} \in \Mat_n (\RR) \otimes \Mat_{m'\times m} (\RR)$, where $\delta'_{i,j} $ denotes now the $m'\times m$ real matrix with $1$ at the $(i,j)$-th position and zeroes elsewhere. Denote by $v'_i$ the column vector of $m'$ real coordinates with $1$ at the $i$-th place and zeroes elsewhere. By similar computations as before, we see that
\[
Q\mathbbm{x} =
\sum_i^{m'}
	\left(\sum_j^m Q^{i,j} g_j\right)
\otimes v'_i=\sum_i^{m'} M_i \otimes v'_i.
\]
Hence,
\begin{multline*}
	M = \sum_i^{m'}M_i^*M_i=\sum_{i,j}^{m'} M_i^*M_j \otimes\left(v'_i \right)^Tv'_j\\=\left(\sum_i^{m'} M_i \otimes v'_i \right)^*
	\left( \sum_j^{m'} M_j \otimes v'_j \right)=\mathbbm{x}^* Q^* Q \mathbbm{x}.
\end{multline*}
\end{proof}

By the above lemma finding a sum of squares decomposition for
$M\in \Mat_n (\RR G)$ is equivalent to showing the existence of a certain positive semi-definite
matrix $P\in \Mat_n (\RR) \otimes \Mat_m (\RR)$ satisfying relation \eqref{eqn:x'Px}.
We will now show that this relation amounts to a (finite) set of linear constraints on the
entries of $P$. Let us first fix some notation. For $P$ we will write
\[
P = \sum_{i,j} \delta_{i,j} \otimes P^{i,j}.
\]
We will be identifying $\RR E$ (where $E = \{g_1, \ldots, g_m \} \subset G$) with $\RR^m$ and therefore $\Mat_m (\RR)$ with $\Mat_E (\RR)$.
Let us define $A_g \in \Mat_E (\RR)$ as follows:
\[
\left(A_g \right)_{x,y} =
\begin{cases}
	1 & \text{if $x^{-1} y = g$}\\
	0 & \text{otherwise}.
\end{cases}
\]

\begin{lemma}\label{lemma:matrixProblemTranslation}
Let $E$, $v$ and $\mathbbm{x} = I_n \otimes v $ be as in \cref{lemma:sdp_matrix_sos}.
Suppose that $M = [m_{i,j}] \in \Mat_n (\RR G)$ is supported on $E^*E$ and
$P \in \Mat_n (\RR) \otimes \Mat_m (\RR)$.
Then $M = \mathbbm{x}^* P \mathbbm{x}$ if and only if for every $g \in G$
\[
  m_{i,j}(g) = \left \langle A_g, P^{i,j} \right \rangle,
\]
where, for any $\xi=\sum\xi_gg\in\RR G$ and $g\in G$, $\xi(g)$ denotes the coefficient $\xi_g$, and $\langle X, Y \rangle = \tr(X^*Y)$ denotes the standard scalar product on $\Mat_m(\RR)$.
\end{lemma}
\begin{proof}
We note that $A_g \equiv 0$ whenever $g \notin E^*E$, hence it's enough
to check the condition for $g \in E^*E$ only.
The proof follows by straightforward calculation:
\begin{multline*}
	m_{i,j}(g) =
	\left(\mathbbm{x}^* P \mathbbm{x} \right)_{i,j} (g) =
	\left(
		\mathbbm{x}^* \left(\delta_{i,j} \otimes P^{i,j} \right) \mathbbm{x}
	\right) (g)\\ =
	\left(v^* P^{i,j} v\right)(g) =
	\sum_{x^{-1} y = g} \left(P^{i,j} \right)_{x,y} =
	\left\langle A_g, P^{i,j} \right\rangle. \qedhere
\end{multline*}
\end{proof}
Denoting by $\delta_{ij}$ the Kronecker delta, we get the following.
\begin{corollary}\label{eqn:optimization_problem}
	Let $M=\left[m_{i,j}\right] \in \Mat_n(\RR G)$. The matrix $M-\lambda I_n$ is a sum of squares for $\lambda\in\RR$ if and only if there exists a finite subset $E\subseteq G$, $|E|=m$, for which the set
	\begin{align*}
	F(M,\lambda)=\Big\{P \in \Mat_{nm}(\RR)&\cong\Mat_{n}(\RR)\otimes\Mat_E(\RR)\;\big\vert\; \\&P\text{ -- positive seimi-definite},\\&m_{i,j}(g)-\delta_{ij}\lambda=
	\left\langle A_g,P^{i,j}\right\rangle \text{ for every }g\in G\Big\}
	\end{align*}
is non-empty.

Since the conditions in the definition of $F(M,\lambda)$ depend on $\lambda$ linearly, it follows that maximal such $\lambda$ can be obtained as the solution to the (convex) optimization problem:
\begin{equation}\label{eqn:optimization_problem2}
	\begin{aligned}
		\text{maximize: }&&&\lambda\\
		\text{subject to: }
		&&& F(M,\lambda)\neq\emptyset.
	\end{aligned}
\end{equation}
\end{corollary}
\begin{proof}
\Cref{lemma:sdp_matrix_sos,lemma:matrixProblemTranslation} reduce the problem of
finding a sum of squares decomposition for $M$ to finding a
positive semi-definite matrix $P$ satisfying a set of linear relations.
Since for every choice of finite $E$ almost all of $A_g \equiv 0$,
the set of constraints is finite and the problem is well-defined.
\end{proof}

\subsection{Certification argument}\label{subsection:certification}
The argument to obtain a mathematical proof of the existence of a sum of squares
decomposition from an inexact solution of the optimization problem
\eqref{eqn:optimization_problem} is similar to that described in \cite{KNO2019},
sections 4.3 and 5. We describe the procedure here very briefly.

Suppose that we numerically solved problem \eqref{eqn:optimization_problem}
and found an inexact solution $(P,\lambda)$,
that is $\Delta_1 - \lambda I_n \approx \mathbbm{x}^*P\mathbbm{x}$.
Since $P$ is obtained numerically it is e.g. positive semi-definite only to numerical precision
(i.e. it may have negative eigenvalues, which are close to $0$)
hence it may not represent a sum of squares as in \cref{lemma:sdp_matrix_sos}.
Therefore we compute first $Q$, a \emph{real} square root of $P$
to guarantee that $Q^* Q$ is positive semi-definite, and find the approximation error
\[r = \Delta_1 - \lambda I_n - \mathbbm{x}^{*} Q^* Q \mathbbm{x}.\]
Since $r$ is $*$-invariant, by \cref{proposition:positiveOrderUnit},
we can dominate it by $\| r\|_1 I_n$ and therefore
\[
\Delta_1 - \left(\lambda - \| r\| _1\right)I_n =
	\mathbbm{x}^* Q^* Q \mathbbm{x} + r + \| r\| _1 I_n \in \Sigma^2\Mat_n (\RR G).
\]

In order to ensure a mathematically rigorous proof, we evaluate the residual
in \emph{interval arithmetic}, i.e. with narrow interval $\overline{\lambda}$
which encloses $\lambda$ and with an interval matrix $\overline{Q}$ which entries
enclose the corresponding entries of $Q$.
Then the $\ell_1$-norm of the residual is an interval $\|\overline{r}\|_1$ and
for any $\lambda_0 \leqslant \inf \left(\overline \lambda - \|\overline{r}\|_1\right)$
we can mathematically prove that $\Delta_1 - \lambda_0 I_n \in \Sigma^2\Mat_n (\RR G)$.

\subsection{The computation}

\begin{proposition}\label{proposition:certifiedSOS}
	The expression
	\[
		d_0d_0^*+\sum_{i,j,k}\left(
			J\left(r_{{i,j,k}}\right)^* J\left(r_{{i,j,k}}\right)+\left(J\left(r'_{{i,j,k}}\right)\right)^*J\left(r'_{{i,j,k}}\right)
		\right) - 0.32I_6
	\]
	is a sum of squares.
\end{proposition}
The proof follows by combining the results
contained in \cref{section:coneOfSoS,section:SoSandOptimization}.
An implementation of the arguments in Julia \cite{bezanson2017julia} programming language
which was used to obtain a numerical approximation of the sum of squares decomposition
is available at \cite{Kaluba_LowCohomologySOS}.

\begin{remark}\normalfont
We excluded relation $r$ for computational reasons using lemma \ref{lemma:addingRelations}.
Indeed, since this relation is the longest among all relations
in the presentation \eqref{eqn:SL3Z_presentation} of $\SLZ$,
excluding it allowed to perform computations on a smaller ball in $\RR G$
than if the relation $r$ would have been included.
\end{remark}

The proposition above and \cref{lemma:addingRelations} yield the following corollary.
\begin{corollary}\label{corollary:certifiedDelta_1}
Let $\SLZ$ be presented as in \cref{eqn:SL3Z_presentation} and
let $\Delta_1$ be the corresponding cohomological Laplacian of \cref*{eqn:Laplacian}.
Then
\[
	\Delta_1-0.32I_6
\]
is a sum of squares.
\end{corollary}

\bibliographystyle{acm}
\bibliography{references}

\end{document}